\documentclass{amsart}
\usepackage{epsfig}
\newtheorem{theorem}{Theorem}[section]
\newtheorem{lemma}[theorem]{Lemma}
\newtheorem{corollary}[theorem]{Lemma}
\theoremstyle{definition}
\newtheorem{definition}[theorem]{Definition}

\newtheorem{notation}[theorem]{Notation}
\theoremstyle{remark}
\newtheorem{remark}[theorem]{Remark}

\numberwithin{equation}{section}

\begin{document}

\title{A Fractal Operator Associated to Bivariate Fractal Interpolation Functions}




\author{S. Verma}
\address{Department of Mathematics, IIT Delhi, New Delhi, India 110016 }
\email{saurabh331146@gmail.com}
\author{P. Viswanathan}
\address{Department of Mathematics, IIT Delhi, New Delhi, India 110016}

\email{viswa@maths.iitd.ac.in}





\begin{abstract}

A general framework to construct fractal interpolation surfaces (FISs) on rectangular grids was presented and bilinear FIS was deduced by  Ruan and Xu [Bull. Aust. Math. Soc. 91(3), 2015, pp. 435-446]. From the view point of operator theory and the stand point of developing some approximation aspects, we revisit the aforementioned construction to obtain a fractal analogue of a prescribed continuous function defined on a rectangular region in $\mathbb{R}^2$. This approach leads to a bounded linear operator  analogous to the  so-called $\alpha$-fractal operator associated with the univariate fractal interpolation function. Several elementary properties of this bivariate fractal operator are reported. We extend the fractal operator to the $\mathcal{L}^p$-spaces for $1 \le p < \infty$. Some approximation aspects of the bivariate continuous fractal functions are also discussed.

\end{abstract}

\maketitle


.

\section{Introduction and Preliminaries}
To fulfill a preparatory role, we shall take a cursory look at the requisite  basic concepts in the theory of fractal interpolation and approximation. Our discussion will be interspersed with an, albeit incomplete, list of references.

Assume that $N>2$ and $\{(x_i,y_i)\in \mathbb{R}^2:i=0,1, \dots, N\}$ is such that  $x_0<x_2<\dots<x_N$. Set $I=[x_0,x_N]$ and for $i=1,2,\dots, N$, let $L_i:I \to [x_{i-1},x_i]$ be a contractive homeomorphism such that
$$L_i(x_0)=x_{i-1}, \quad L_i(x_N)=x_i.$$
Let $F_i:I \times \mathbb{R} \to \mathbb{R}$ be a continuous map which satisfies the following
\begin{enumerate}
\item $F_i(x_0,y_0)= y_{i-1}$ and $F_i(x_N, y_N)= y_i,$
\item $|F_i(x,y)-F_i(x,y')|\leq r_i|y-y'|$ for all $x \in I$ and  $y,y' \in \mathbb{R}$, where $r_i \in [0,1).$
\end{enumerate}
 Define $W_i(x,y)=\big(L_i(x), F_i(x,y)\big)$ for $(x,y) \in I \times \mathbb{R}$. Then, $\{I \times \mathbb{R}; W_1, W_2,\dots, W_N  \}$ is an \emph{iterated function system}, IFS for short \cite{H}. Barnsley \cite{MF1} proved that its unique invariant set $G$ is the graph of a continuous function $g:I \to \mathbb{R}$, referred to as  a \emph{fractal interpolation function} (FIF) such that for every $i=0,1, \dots,N$ and $x \in I$, we have $g(x_i)=y_i$ and $g$ satisfies a self-referential equation  $$g(L_i(x))=F_i\big(x,g(x)\big), \quad i=1,2,\dots,N.$$ As a new class of interpolants, FIFs demonstrated  more advantages than the classical interpolants  in fitting and approximation of naturally occurring functions that display some kind of self-similarity. Consequently, numerous papers on this subject have been published so far and we refer the reader to \cite{BM,AKBCK,Luor,Ri} for some discussions, recent results and further references.

\par Most widely studied FIFs are obtained by considering each $L_i$ and $F_i$ in the following form.
$$L_i(x) = a_i x +b_i, \quad F_i(x,y)=\alpha_i y + q_i(x), $$
where $q_i: I \to \mathbb{R}$ is suitable continuous map and $\alpha_i \in \mathbb{R}$ is such that $|\alpha_i|<1$. Let $f : I \to \mathbb{R}$ be a prescribed continuous function.
By taking $$q_i(x) = f(L_i(x))-\alpha_i b(x),$$ where  $b$ is an appropriate continuous map, Navascu\'es \cite{M2} observed that fractal analogue of a given continuous function $f$ can be constructed. The FIF associated with such a IFS is called the \emph{$\alpha$-fractal function} for $f$ and is denoted by $f_{\Delta, b}^{\alpha}$, where $\alpha=(\alpha_1, \alpha_2, \dots, \alpha_N) \in (-1,1)^{N}$ is referred to as the \emph{scale vector}. Such an approach leads to an operator $$\mathcal{F}_{\Delta, b}^{\alpha}: f \mapsto f_{\Delta, b}^{\alpha},$$ which depends on a chosen partition $\Delta$, function $b$ and constants $\alpha_i$, and which sends each $f$ to its fractal perturbation  $f_{\Delta, b}^{\alpha}$. If $b=Lf$, where $L$ is a bounded linear operator, then the  operator $\mathcal{F}_{\Delta, L}^{\alpha}$ is a bounded linear operator termed the \emph{fractal operator}. This operator formulation of fractal functions somewhat hidden in the construction of FIFs enables them to interact with other traditional branches of mathematics including operator theory, complex analysis,
harmonic analysis and approximation theory \cite{M2,M1,M6,M5,NC,M9}. More recently, the second author and
collaborators identified suitable values of the parameters so that the $\alpha$-fractal function $f_{\Delta,b}^\alpha$ preserves
the shape properties inherent in the source function $f$ \cite{M10}. This lead to the intersection of
 the two different traditions - theory of FIF and univariate constrained approximation -
that were otherwise developing independently.
\par
 In the attempt to different extensions of FIF, it is natural to seek FIFs in higher-dimensional cases, in particular, the two-dimensional case aiming
at a realistic modeling of rough surfaces. While it is straightforward to define a similar IFS as that in the one-dimensional case, it is hard to ensure that the invariant set of such IFS is the graph of a continuous function. There are various approaches for the construction
of fractal interpolation surface (FIS) (an, albeit incomplete list of references \cite{D,AKBCK2,Dalla,PM,Ruan}), each with their particular strengths and weaknesses. More recently, we have developed a few constructive approaches
for solving constrained interpolation by fractal surfaces \cite{V4,V5}. However, a unified approach for
developing fractal versions of various traditional methods of constrained bivariate interpolation is
strongly felt.
\par
To summarize, the approach to the univariate FIF wherein a source function is perturbed to obtain its  fractal analogue and the fractal operator emerged thereby
 (i) enabled FIF to interact with other
branches of mathematics (ii) played a key role in developing fractal versions of fundamental
theorems in constrained approximation (iii) provided a unified approach to various shape
preserving fractal interpolation schemes (iv) reinformed the ubiquity of fractal functions as claimed by the fractal researchers. On the other hand, to the best of our knowledge, there
is no research reported in the direction of  bivariate analogue of the $\alpha$-fractal operator arising from FISs. The reason perhaps is that a general framework for the construction of FISs was missing until the researches reported in \cite{Ruan}.  Motivated by these
facts, in broad terms, the object of the current paper is to provide foundational aspects of
a bivariate version of the $\alpha$-fractal function and fractal operator. To be precise, the current paper is a realization that
the general framework to generate FIS given in \cite{Ruan} can be used to obtain a fractal analogue of a continuous real-valued function defined on a rectangular region in $\mathbb{R}^2$.

\par
Let $-\infty<x_0<x_N< \infty$, $-\infty<y_0<y_M< \infty$, $I=[x_0,x_N]$ and $J=[y_0,y_M]$. As is customary, we denote by  $\mathcal{C}(I \times J, \mathbb{F})$, the Banach space of all continuous functions $f: I \times J \to \mathbb{F}$ supplied with the uniform norm, where $\mathbb{F}$ is the real field $\mathbb{R}$ or the complex field $\mathbb{C}$. In Section \ref{BFOsecb}, we shall define and establish some elementary properties of the fractal operator $\mathcal{F}^{\alpha}_{\triangle,L}: \mathcal{C}(I \times J, \mathbb{R}) \rightarrow \mathcal{C}(I \times J, \mathbb{R})$ that maps a bivariate continuous function $f(x,y)$ to its fractal analogue $f^{\alpha}_{\triangle,L}(x,y)$. Using the usual density arguments, we extend our fractal operator to $\mathcal{L}^p(I \times J, \mathbb{C})$ in Section \ref{BFOsecc}. Some approximation aspects of the continuous bivariate $\alpha$-fractal functions
and bivariate ``fractal polynomials" are investigated in Section \ref{BFOsecd}.   Our approach reveals a natural kinship with the research works of Navascu\'{e}s in the field of univariate $\alpha$-fractal functions scattered in the literature; see, for instance, \cite{M1,M2,M5}. However, we feel that the new approach to bivariate fractal functions introduced herein with only the surface being scratched, finds potential applications in approximation problems. Let us emphasize that the bivariate version is proposed
not merely as an extension of the univariate case; but with an eye towards extending fractal surfaces to the
territory of constrained approximation. We envisage  that  the constrained approximation with fractal surfaces is a
problem involving larger resources than their traditional counterparts.

\section{Auxiliary Apparatus}
In this section, we revisit a general framework to construct FISs on rectangular grids; for details the reader is referred to \cite{Ruan}.\\
Let $I=[x_0,x_N]$ and $J=[y_0,y_M].$ Suppose that an interpolation data set \\$\{(x_i,y_j,z_{ij}) \in \mathbb{R}^3 : i=0,1,\dots,N; j=0,1,\dots,M\}$ such that $x_0< x_1 < \dots <x_N$ and $y_0<y_1<\dots <y_M$  is provided. Following the notation in \cite{Ruan}, we write $\Sigma_N=\{1,2,\dots,N\},$ $ \Sigma_{N,0}=\{0,1,\dots N \},$ $\partial \Sigma_{N,0}=\{0,N\} $ and int$\Sigma_{N,0}=\{1,2,\dots,N-1\}.$ Similarly, we can define $\Sigma_M, \Sigma_{M,0}, \partial \Sigma_{M,0}$ and int$\Sigma_{M,0}.$ Let $I_i=[x_{i-1},x_i]$ and $J_j=[y_{j-1},y_j]$ for $i \in \Sigma_N$ and $j \in \Sigma_M.$ For any $i \in \Sigma_N,$ let $u_i:I \rightarrow I_i$ be a contractive homeomorphism satisfying
\begin{equation*}
\left\{
   \begin{split}
   &~ u_i(x_0)=x_{i-1}, \quad u_i(x_N)=x_i, ~~\text{if $i$ is odd},\\
        &~ u_i(x_0)=x_i, \quad u_i(x_N)=x_{i-1},~~ \text{if $i$ is even},~~\text{and} \\
          &~   |u_i(x_1)-u_i(x_2)| \le \alpha_i|x_1 -x_2|, ~~~~ \forall ~~x_1, x_2 \in I,
             \end{split}\right.
               \end{equation*}
where $0 < \alpha_i < 1$ is a given constant. Similarly, for any $j \in \Sigma_M,$ let $v_j:J \rightarrow J_j$ be a contractive homeomorphism satisfying

  \begin{equation*}
                \left\{
   \begin{split}
                &~ v_j(y_0)=y_{j-1}, \quad v_j(y_M)=y_j,\text{ if $j$ is odd},\\
                 &~ v_j(y_0)=y_j, \quad v_j(y_N)=y_{j-1}, ~~ \text{if $j$ is even},\\
                &~ |v_j(y_1)-v_j(y_2)| \le \beta_j|y_1 -y_2|, ~~~~ \forall ~~y_1, y_2 \in J,
                 \end{split}\right.
                   \end{equation*}
where $0 < \beta_j < 1$ is a given constant. By the definitions of $u_i$ and $v_j,$ it is easy to check that
\begin{equation*}
                \begin{aligned}
                 u^{-1}_i(x_i)=u^{-1}_{i+1}(x_i), ~~~~\forall ~~i \in \text{int}\Sigma_{N,0},
                 \end{aligned}
                   \end{equation*}
and
 \begin{equation*}
                 \begin{aligned}
                 v^{-1}_j(y_j)=v^{-1}_{j+1}(y_j),~~~~ \forall ~~j \in \text{int}\Sigma_{M,0}.
                  \end{aligned}
                    \end{equation*}
Let $\tau: \mathbb{Z}\times \{0,N,M\} \rightarrow \mathbb{Z}$ be defined by
\begin{equation*}
\tau(i,0)=
                   \begin{cases} i-1,~~~~~~~~ \tau(i,N)=\tau(i,M)=i,  \text{ if $i$ is odd}\\
                      i,~~~~~~~~ \tau(i,N)= \tau(i,M)=i-1, \text{ if $i$ is even}.
                  \end{cases}
                    \end{equation*}
 Let $K=I \times J \times \mathbb{R}.$ For each $(i,j) \in \Sigma_N \times \Sigma_M,$ let $F_{ij}:K \rightarrow \mathbb{R}$ be a continuous function satisfying $$ F_{ij}(x_k,y_l,z_{kl})=z_{\tau(i,k),\tau(j,l)},  ~~~~~~ \forall~~ (k,l) \in \partial \Sigma_{N,0} \times  \partial \Sigma_{M,0}, \text{and}$$  $$|F_{ij}(x,y,z')- F_{ij}(x,y,z'')| \le \gamma_{ij} |z' - z''|, ~~~~~~ \forall~~ (x,y) \in I \times J ,~~\text{and}~~ z',z'' \in \mathbb{R},$$ where $0 < \gamma_{ij} < 1$ is a given constant.\\
 Finally, for each $(i,j) \in \Sigma_N \times \Sigma_M,$ we define $W_{ij}:K \rightarrow I_i \times J_j \times \mathbb{R}$ by $$ W_{ij}(x,y,z)=\big(u_i(x),v_j(y),F_{ij}(x,y,z)\big).$$
Then $\{K, W_{ij}: (i,j) \in\Sigma_N \times \Sigma_M \}$ is an IFS.

\begin{theorem} \cite{Ruan} \label{BFOTHM1}
Let $\{K, W_{ij}: (i,j) \in \Sigma_N \times \Sigma_M \}$ be the IFS defined as above. Assume that the map $F_{ij}$, $(i,j) \in \Sigma_N \times \Sigma_M$ satisfies the following matching conditions
\begin{enumerate}
\item  for all $i \in \text{int}\Sigma_{N,0}, j \in \Sigma_M$ and $x^*=u^{-1}_i(x_i)=u^{-1}_{i+1}(x_i),$
$$F_{ij}(x^*,y,z)=F_{i+1,j}(x^*,y,z),~~  \forall~ y \in J, z \in \mathbb{R}, ~\text{and} $$
 \item for all $i \in \Sigma_N, j \in \text{int} \Sigma_{M,0}$ and $y^*=v^{-1}_j(y_j)=v^{-1}_{j+1}(y_j),$ $$F_{ij}(x,y^*,z) = F_{i,j+1}(x,y^*,z),~~ \forall~ x \in I, z \in \mathbb{R}.$$
     \end{enumerate}
Then there exists a unique continuous function $f:I \times J \rightarrow \mathbb{R}  $ such that $f(x_i,y_j)=z_{ij}$ for all $(i,j) \in \Sigma_{N,0} \times \Sigma_{M,0}$ and $G= \cup_{(i,j) \in \Sigma_{N} \times \Sigma_{M}} W_{ij}(G),$ where $G=\big\{(x,y,f(x,y)):(x,y) \in I \times J \big\}$ is the graph of $f.$
\end{theorem}
\begin{definition}
We call $G$ in the aforementioned theorem as the FIS and $f$ as the bivariate FIF with respect to the IFS $\big\{K, W_{ij}: (i,j) \in \Sigma_N \times \Sigma_M \big\}$ corresponding to the data set $\{(x_i,y_j,z_{ij}) \in \mathbb{R}^3 : i=0,1,\dots,N; j=0,1,\dots,M\}$.
\end{definition}
\begin{remark} \label{BFOaddrem1}
Consider the set $$\mathcal{C}^*(I \times J, \mathbb{R}):= \Big\{ g\in \mathcal{C}(I \times J, \mathbb{R}): g(x_i,y_j)=z_{ij}~\forall~~ (i,j)\in \Sigma_{N,0} \times \Sigma_{M,0} \Big\}$$
endowed with the supremum metric. Let us define an operator $T$, referred to as Read-Bajraktarevi\'{c} operator
$$Tg(x,y)= F_{ij} \Big(u_i^{-1}(x), v_j^{-1}(y), g\big(u_i^{-1}(x),v_j^{-1}(y) \big)\Big),~ (x,y) \in I_i \times J_j,~ (i,j) \in \Sigma_N \times \Sigma_M.$$
The bivariate FIF $f$ in the previous definition is the unique fixed point of $T$. Consequently, $f$ satisfies the self-referential equation:
$$f(x,y) = F_{ij} \Big(u_i^{-1}(x), v_j^{-1}(y), f\big(u_i^{-1}(x),v_j^{-1}(y) \big)\Big),~ (x,y) \in I_i \times J_j,~ (i,j) \in \Sigma_N \times \Sigma_M.$$
\end{remark}
Under some assumptions on the elements in the IFS involved, the box counting dimension of the graph of the resulting bilinear FIS is studied in \cite{Ruan2}, which we shall recall here. Let $M=N$. Let $g$ be the bilinear function on $I \times J$ satisfying $g(x_i,y_j)=z_{ij}$ for all $(i,j) \in \partial \Sigma_{N,0} \times \partial \Sigma_{N,0}$. That is
\begin{equation*}
\begin{split}
g(x,y)= &~ \frac{1}{(x_N-x_0)(y_N-y_0)} \Big[(x_N-x)(y_N-y) z_{00} + (x-x_0)(y_N-y)z_{N0}\\&~+(x_N-x)(y-y_0)z_{0N}+ (x-x_0)(y-y_0)z_{NN}\Big].
\end{split}
\end{equation*}
Let $h: I \times J \to \mathbb{R}$ be a function satisfying $h(x_i,y_j)=z_{ij}$ for all $(i,j) \in \Sigma_{N,0} \times \Sigma_{M,0}$ such that $h$ when restricted to $D_{ij}$ is bilinear for all $(i,j) \in \Sigma_{N} \times \Sigma_{N}$. Assume that $\{s_{i,j}: (i,j) \in \Sigma_{N,0} \times \Sigma_{N,0}\}$ is a given data set with $|s_{ij}|<1$ for all $i,j$. We define $S: I \times J \to \mathbb{R}$ such that $S(x_i,y_j)=s_{i,j}$ for all $(i,j) \in \Sigma_{N,0} \times \Sigma_{M,0}$ and that $S$ restricted to $D_{ij}$ is bilinear for all $(i,j) \in \Sigma_{N} \times \Sigma_{M}$. Define
$$F_{ij}(x,y,z) = S \big(u_i(x),v_j(y)\big) \big(z-g(x,y)\big)+ h\big(u_i(x),v_j(y)\big).$$

Assume that the scaling factors are steady, that is, for each $(i,j) \in \Sigma_N \times \Sigma_N$ all of $s_{i-1,j-1}$, $s_{i,j-1}$, $s_{i-1,j}$ and $s_{i,j}$ are nonnegative or all of them are nonpositive.  Let
\begin{equation*}
\begin{split}
&~ \sum_{i,j \in \Sigma_N} \Big |S\big(u_i(x_0), v_j(y_0)\big)  \Big| =\sum_{i,j \in \Sigma_N} \Big |S\big(u_i(x_0), v_j(y_N)\big)  \Big| \\
=&~ \sum_{i,j \in \Sigma_N} \Big |S\big(u_i(x_N), v_j(y_0)\big)  \Big|=\sum_{i,j \in \Sigma_N} \Big |S\big(u_i(x_N), v_j(y_N)\big)  \Big|=: \gamma
\end{split}
\end{equation*}
\begin{theorem} \label{BFOkeythm2} \cite{Ruan2}
Let $f$ be the bilinear FIF determined with aforementioned assumptions on the IFS. If $\gamma>N$ and the interpolation points $\big\{(x_i,y_j,z_{i,j})  \big\}$ are not co-bilinear, then $\dim_B \big(\text{Graph}(f)\big) = 1+ \dfrac{\log \gamma}{\log N}$. Otherwise, $\dim_B \big(\text{Graph}(f)\big)=2.$
\end{theorem}

\section{Associated fractal linear operator on $\mathcal{C}(I \times J, \mathbb{R})$}\label{BFOsecb}

Let $I=[x_0,x_N]$ and $J=[y_0,y_M].$ Let $f: I\times J \rightarrow \mathbb{R}$ be a given continuous function. Define a net $\Delta$ by \begin{equation*}
                 \begin{aligned}
                   x_0< x_1 < \dots <x_N; \\ y_0<y_1<\dots <y_M.
        \end{aligned}
                    \end{equation*}
Let $L: \mathcal{C}(I \times J, \mathbb{R}) \rightarrow \mathcal{C}(I \times J, \mathbb{R})$ be a bounded linear operator satisfying $Lf \neq f$, $$(Lf)(x_i,y_j)=f(x_i,y_j), ~~~~\forall~ (i,j) \in \partial \Sigma_{N,0} \times \partial \Sigma_{M,0}.$$  Let $\alpha: I\times J \rightarrow \mathbb{R}$ be a continuous function such that $$\|\alpha\|_{\infty}:= \sup\big\{|\alpha(x,y)|:(x,y) \in I \times J \big\} < 1.$$
Set $K= I \times J \times \mathbb{R} $ and define $F_{ij}: K \rightarrow \mathbb{R} $ by
\begin{equation} \label{BFOreeq1}
F_{ij}(x,y,z)= \alpha\big(u_i(x),v_j(y)\big)z+f\big(u_i(x),v_j(y)\big)- \alpha\big(u_i(x),v_j(y)\big)(Lf)(x,y),
\end{equation}
where $u_i \in \mathcal{C}(I,\mathbb{R})$ and $v_j \in \mathcal{C}(J, \mathbb{R})$ satisfy conditions prescribed in the previous section. In the sequel, we define $u_i$ and $v_j$ to be linear functions satisfying the required conditions, say
\begin{equation} \label{BFOreeq2}
u_i(x)=a_i x +b_i, \quad v_j(y)= c_j y + d_j,
\end{equation}
where constants involved are suitably determined.
For each $(i,j) \in \Sigma_N \times \Sigma_M,$ we define $W_{ij}:K \rightarrow I_i \times J_j \times \mathbb{R}$
\begin{equation} \label{BFOreeq3}
W_{ij}(x,y,z)=\big(u_i(x),v_j(y),F_{ij}(x,y,z)\big).
\end{equation}
Let us mention  two examples for such an operator $L: \mathcal{C}(I\times J, \mathbb{R}) \rightarrow \mathcal{C}(I \times J, \mathbb{R}).$
\begin{enumerate}
\item $(Lf)(x,y)=f(x,y) t(x,y),$ where  $t \in \mathcal{C} (I \times J, \mathbb{R})$ is a fixed  non-constant function such that $t(x_i,y_j)=1, ~\forall~(i,j) \in \partial \Sigma_{N,0} \times \partial \Sigma_{M,0}.$ On calculating the operator norm of $L,$ we obtain $\|L\|= \|t\|_{\infty}.$
    \item $(Lf)(x,y)=(f\circ t)(x,y),$ where $t \in  \mathcal{C} (I \times J, I \times J)$ is a fixed map $t \neq Id$, the identity map and $t(x_i,y_j)=(x_i,y_j),~\forall~(i,j) \in \partial\Sigma_{N,0} \times \partial \Sigma_{M,0}.$ In this case, we get $\|L\|=1.$
\end{enumerate}
It is straightforward to see that $F_{ij}$ satisfies the  matching conditions required in Theorem \ref{BFOTHM1} and therefore we have the following.
\begin{theorem}
Let $\big\{K, W_{ij}: (i,j) \in\Sigma_N \times \Sigma_M \big\}$ be the IFS defined through (\ref{BFOreeq1})-(\ref{BFOreeq3}) above.
Then there exists a unique continuous function $f^{\alpha}_{\Delta,L}:I \times J \rightarrow \mathbb{R}  $ such that $f^{\alpha}_{\Delta,L}(x_i,y_j)=f(x_i,y_j)$ for all $(i,j) \in \Sigma_{N,0} \times \Sigma_{M,0}$ and $G= \cup_{(i,j) \in \Sigma_{N} \times \Sigma_{M}} W_{ij}(G),$ where $G=\big\{(x,y,f^{\alpha}_{\Delta,L}(x,y)):(x,y) \in I \times J \big\}$ is the graph of $f^{\alpha}_{\Delta,L}.$
\end{theorem}
\begin{remark}
Being the fixed point of the RB-operator (Cf. Remark \ref{BFOaddrem1}), $f^{\alpha}_{\triangle,L}$ satisfies the self-referential equation:
 $$ f^{\alpha}_{\Delta,L}(x,y)=F_{ij}\Big(u_i^{-1}(x),v_j^{-1}(y),f^{\alpha}_{\Delta,L}(u_i^{-1}(x),v_j^{-1}(y))\Big),~~\forall~~ (x,y) \in I_i \times J_j.$$
 That is, for all $(x,y) \in I_i \times J_j,$ where $(i,j)  \in\Sigma_N \times \Sigma_M$ we have
 \begin{equation} \label{fnleq}
  f^{\alpha}_{\Delta,L}(x,y)= f(x,y)+\alpha(x,y) f^{\alpha}_{\Delta,L}(u_i^{-1}(x),v_j^{-1}(y))- \alpha(x,y)(Lf)(u_i^{-1}(x),v_j^{-1}(y)).
 \end{equation}
 \end{remark}
 The function $f^{\alpha}_{\Delta,L}$ appeared in the previous remark is important enough to be dignified with a name of its own.
 \begin{definition}
 We call the aforementioned self-referential function $f^{\alpha}_{\Delta,L}$ as (bivariate) $\alpha$-fractal function, fractal perturbation or associate fractal function corresponding to $f$ with respect to the operator $L$ and the net $\Delta$.
 \end{definition}
 \begin{remark} \label{BFOrerem1}
  Theorem \ref{BFOkeythm2} establishes the box dimension of the graph of $f^{\alpha}_{\Delta,L}$ for some special class of functions $f$ and choice of parameters $\alpha$, $\Delta$ and $L$. Recently, the box counting dimension of the graph of
 a univariate $\alpha$-fractal function established in a more general setting in \cite{AGN}. We believe that by modifying and adapting these results, the box dimension of the graph of $f^{\alpha}_{\Delta,L}$ can be computed for a more general class and details will appear elsewhere.
 \end{remark}
\begin{definition}
  For a fixed net $\Delta$, a scale function $\alpha$ and an operator $L$, let us  define  the $\alpha$-fractal operator or simply fractal operator $$\mathcal{F}^{\alpha}_{\triangle,L}: \mathcal{C}(I \times J, \mathbb{R}) \rightarrow \mathcal{C}(I \times J,\mathbb{R}), \quad  \mathcal{F}^{\alpha}_{\triangle,L}(f)= f^{\alpha}_{\triangle,L}.$$
 \end{definition}
Next let us recall a pair of lemmas and  definitions that are fundamental in functional analysis; see, for instance, \cite{BB}.
 \begin{lemma}\label{BFOel1}
 If $T$ is a bounded linear operator from Banach space into itself such that $\|T\| < 1,$ then $ (Id - T)^{-1}$ exists and it is bounded.
 \end{lemma}
 \begin{lemma}\label{BFOel2}
  If $T$ is a bounded linear operator and $S$ is a compact operator on a normed linear space $X,$ then $TS$ and $ST$ are compact operators.
  \end{lemma}
  Following \cite{BB}, we shall use the product notation for the value of a linear functional on an element: $\langle x,f \rangle = \langle f,x \rangle = f(x)$ for $x$ in a normed linear space $X$ and $f$ in $X^*$, the dual of $X$. Let $X,Y$ be normed spaces and  $T: X \to Y$ be a bounded linear operator. The adjoint or dual $T^*$ of $T$ is the unique map $T^* : Y^* \to X^*$ such that
  $$ \langle x, T^*g \rangle = \langle Tx, g \rangle, ~~ \forall~ x \in X, ~~g \in Y^*.$$
 \begin{definition}
    An operator $T$ is Fredholm if:
    \begin{enumerate}
    \item  Range$(T)$ is closed.
    \item  $\ker(T)$ and $\ker(T^*)$ are finite-dimensional.
    \end{enumerate}
     Further, the index of a Fredholm operator is defined as $$ \text{index} (T)= \dim(\ker(T))- \dim (\ker (T^*)).$$
    \end{definition}
    \begin{definition}
    Given a Banach space $X$, the annihilator of a subspace $L$ of $X^*$ in $X$ (or the preannihilator of $L$) is
  $$^a L=\big\{ x \in X: \langle x, f \rangle =0 ~~\forall~~f \in L\big\}.$$
    \end{definition}
The following theorem exhibits some elementary properties of the bivariate $\alpha$-fractal function and the corresponding fractal operator.
 This result is reminiscent of the univariate case scattered in the fractal literature; see, for instance, \cite{M1}.
 However, for the sake of completeness and record, we provide a fairly self-contained arguments.
 \begin{theorem}\label{BFOTHM2}
 Let $ \| \alpha \|_{\infty} = \sup \big\{|\alpha(x,y)|: (x,y) \in I \times J \big\} $, and let $Id$ be the identity operator on $ \mathcal{C}(I \times J, \mathbb{R})$.
 \begin{enumerate}
 \item For any $f \in \mathcal{C}(I \times J, \mathbb{R})$, the perturbation error satisfies $$ \| f_{\Delta,L}^{\alpha} - f \|_{\infty} \leq \frac{\| \alpha \|_{\infty}}{1- \|\alpha\|_{\infty}} \|f-Lf\|_{\infty}.$$ In particular, if $\alpha=0$, then $\mathcal{F}_{\Delta, L}^{\alpha}=Id.$ \\
 \item The fractal operator $\mathcal{F}_{\Delta, L}^{\alpha}$ is a bounded linear operator with respect to the uniform norm on $\mathcal{C}(I \times J, \mathbb{R})$. Furthermore, the operator norm satisfies $$ \| \mathcal{F}_{\Delta, L}^{\alpha} \| \leq 1+ \frac{\|\alpha\|_{\infty}~~ \| Id - L \| }{1-\| \alpha \|_{\infty}}.$$
 \item  For $\| \alpha \|_{\infty} < \|L\|^{-1}$, $\mathcal{F}_{\Delta, L}^{\alpha}$ is bounded below. In particular, $\mathcal{F}_{\Delta, L}^{\alpha}$ is one to one.
 \item If $\| \alpha \|_{\infty} < (1 + \|Id - L\|)^{-1}$, then $\mathcal{F}_{\Delta,L}^{\alpha}$ has a bounded inverse and consequently a topological automorphism (i.e., a bijective bounded linear map with a bounded inverse from $\mathcal{C}(I \times J, \mathbb{R})$ to itself). Moreover, $$\| (\mathcal{F}_{\Delta,L}^{\alpha})^{-1}\| \leq \frac{1+\| \alpha\|_{\infty} }{1-\| \alpha\|_{\infty} \|L\|} .$$
 \item  If $\| \alpha\|_{\infty} \ne 0,$ then the fixed points of $L$ are  the fixed points of the fractal operator $\mathcal{F}_{\Delta, L}^{\alpha}$ as well.
 \item If $1$ belongs to the point spectrum of $L$, then $ 1 \le \|\mathcal{F}_{\Delta,L}^{\alpha}\|.$
 \item For $\| \alpha\|_{\infty} < \|L\|^{-1}$, the fractal operator  $\mathcal{F}_{\Delta,L}^{\alpha}$ is not a compact operator.
 \item If $\| \alpha\|_{\infty} < (1 + \|Id - L\|)^{-1},$ then $ \mathcal{F}_{\Delta,L}^{\alpha}$ is Fredholm and its index is $0.$
\end{enumerate}
 \end{theorem}
 \begin{proof}
 \begin{enumerate}
\item For $(x,y) \in I_i \times J_j,$ from (\ref{fnleq}) we have
  $$ f_{\Delta,L}^{\alpha}(x,y) - f(x,y)= \alpha(x,y) \Big[f^{\alpha}\big(u^{-1}_i(x),v^{-1}_j(y)\big)- (Lf)\big(u^{-1}_i(x),v^{-1}_j(y)\big)\Big].$$
Therefore
 $$ |f_{\Delta,L}^{\alpha}(x,y) - f(x,y)|\le \|\alpha\|_{\infty}~~\|f_{\Delta,L}^{\alpha} -Lf\|_{\infty}.$$
 Since the above inequality is true  for all  $(x,y) \in I_i \times J_j$ where $(i,j)\in   \Sigma_{N} \times \Sigma_{M}$, we conclude that
 \begin{equation}\label{BFOeqw}
    \begin{aligned}
       \|f_{\Delta,L}^{\alpha}- f\|_{\infty}\le \|\alpha\|_{\infty}~~\|f_{\Delta,L}^{\alpha} -Lf\|_{\infty}.
    \end{aligned}
    \end{equation}
Using the triangle inequality, we get
 $$ \|f_{\Delta,L}^{\alpha}- f\|_{\infty}\le \|\alpha\|_{\infty}~\Big[\|f_{\Delta,L}^{\alpha} -f\|_{\infty}+\|f -Lf\|_{\infty}\Big].$$
  This demonstrates
  \begin{equation*}
     \begin{aligned}
         \|f_{\Delta,L}^{\alpha}- f\|_{\infty}\le \frac{\|\alpha\|_{\infty} ~~\|Id-L\|}{1-\|\alpha\|_{\infty}}~~\|f\|_{\infty}.
    \end{aligned}
     \end{equation*}
For $\|\alpha\|_{\infty} =0,$ the previous inequality produces $ \|f_{\Delta,L}^{\alpha}- f\|_{\infty} = 0, $ and hence  $f_{\Delta,L}^{\alpha} = f$ for all $f \in \mathcal{C}(I \times J, \mathbb{R})$.  That is, $\mathcal{F}_{\Delta,L}^{\alpha}=Id.$
\item Let $f, g \in \mathcal{C}(I \times J, \mathbb{R})$ and $ \beta , \gamma \in \mathbb{R}.$ For $(x,y) \in I_i \times J_j,$ we have
 $$ \beta f_{\Delta,L}^{\alpha}(x,y) = \beta f(x,y) +  \beta  \alpha(x,y) \Big[f_{\Delta,L}^{\alpha}\big(u^{-1}_i(x),v^{-1}_j(y)\big)- (Lf)\big(u^{-1}_i(x),v^{-1}_j(y)\big)\Big],$$
 $$ \gamma g_{\Delta,L}^{\alpha}(x,y) = \gamma g(x,y) + \gamma  \alpha(x,y) \Big[g_{\Delta,L}^{\alpha}\big(u^{-1}_i(x),v^{-1}_j(y)\big)- (Lg)\big(u^{-1}_i(x),v^{-1}_j(y)\big)\Big].$$
 Adding the above two equations, one gets
 \begin{equation*}
 \begin{split}
  \big(\beta f_{\Delta,L}^{\alpha}+\gamma g_{\Delta,L}^{\alpha}\big)(x,y) = &~ \big(\beta f +\gamma g\big)(x,y) +    \alpha(x,y).\\&~ \big(\beta f_{\Delta,L}^{\alpha}+\gamma g_{\Delta,L}^{\alpha} - L(\beta f+\gamma g)\big)\big(u^{-1}_i(x),v^{-1}_j(y)\big).
  \end{split}
  \end{equation*}
  The previous equation reveals that $\beta f_{\Delta,L}^{\alpha}+\gamma g_{\Delta,L}^{\alpha}$ is the fixed point of RB-operator
  \begin{equation*}
 \begin{split}
  (Th) (x,y) =&~ \alpha(x,y) h\big( u^{-1}_i(x),v^{-1}_j(y) \big)+ (\beta f+ \gamma g)(x,y)\\&~ -\alpha(x,y)L(\beta f+ \gamma g)\big( u^{-1}_i(x),v^{-1}_j(y) \big),
  \end{split}
  \end{equation*}
 Since the fixed point of the RB operator is unique, we have $(\beta f+\gamma g)_{\Delta,L}^{\alpha}= \beta f_{\Delta,L}^{\alpha}+\gamma g_{\Delta,L}^{\alpha},$ which reveals the linearity of the operator $\mathcal{F}_{\Delta,L}^{\alpha}.$
From the previous item we write $$ \|f_{\Delta,L}^{\alpha}\|_{\infty}- \|f\|_{\infty}\le \frac{\|\alpha\|_{\infty} ~~\|Id-L\|}{1-\|\alpha\|_{\infty}}~~\|f\|_{\infty}.$$ It implies that $$ \|\mathcal{F}_{\Delta,L}^{\alpha}(f)\|_{\infty} \le \|f\|_{\infty}+ \frac{\|\alpha\|_{\infty} ~~\|Id-L\|}{1-\|\alpha\|_{\infty}}~~\|f\|_{\infty}.$$
 Therefore $\mathcal{F}_{\Delta,L}^{\alpha}$ is a bounded linear operator.
 \item From (\ref{BFOeqw})
 \begin{equation*}
 \begin{split}
 \|f\|_{\infty}-\|f^{\alpha}_{\Delta,L}\|_{\infty}   \le \|f_{\Delta,L}^{\alpha}- f\|_{\infty}\le &~ \|\alpha\|_{\infty}~~\|f_{\Delta,L}^{\alpha} -Lf\|_{\infty}\\ \le &~ \| \alpha\|_{\infty}~~(\|f_{\Delta,L}^{\alpha}\|_{\infty} +\|L\|~~\|f\|_{\infty}) .
 \end{split}
 \end{equation*}
 Hence we get $(1-\| \alpha \|_{\infty}\|L\|)~~\|f\|_{\infty} \le (1+\| \alpha\|_{\infty})~~\|f_{\Delta,L}^{\alpha}\|_{\infty} .$ If  $\| \alpha \|_{\infty} < \|L\|^{-1} $, then
 \begin{equation}\label{BFOeqz}
       \begin{aligned}
          \|f\|_{\infty} \le \frac{1+\| \alpha\|_{\infty}}{1- \| \alpha\|_{\infty} \|L\|}\|f_{\Delta,L}^{\alpha}\|_{\infty}.
       \end{aligned}
       \end{equation}
   Thus $\mathcal{F}_{\Delta,L}^{\alpha}$ is bounded below.
   \item  From the hypothesis $ \|Id -\mathcal{F}_{\Delta,L}^{\alpha}\| \le \frac{\|\alpha\|_{\infty} ~\|Id-L\|}{1-\|\alpha\|_{\infty}} < 1.$
    Consequently, Lemma \ref{BFOel1} dictates that $\mathcal{F}_{\Delta,L}^{\alpha}$ has a bounded inverse. From (\ref{BFOeqz}), we infer that $$\|(\mathcal{F}_{\Delta,L}^{\alpha})^{-1}(f)\|_{\infty} \le \frac{1+\| \alpha\|_{\infty}}{1- \| \alpha\|_{\infty} \|L\|}\|f\|_{\infty},$$ which in turn  yields the required bound for the operator norm of $ (\mathcal{F}_{\Delta,L}^{\alpha})^{-1}.$
    \item Let $ \| \alpha\|_{\infty} \ne 0,$ and $f$ be a fixed point of $L.$ From (\ref{BFOeqw})   $$ \|f_{\Delta,L}^{\alpha}- f\|_{\infty}\le \|\alpha\|_{\infty}~~\|f_{\Delta,L}^{\alpha} -f\|_{\infty}.$$ Since $\|\alpha\|_{\infty} <1$, this  implies that $ f_{\Delta,L}^{\alpha}=f.$
       \item  Choose $g \in \mathcal{C}(I \times J, \mathbb{R})$ such that $Lg=g$ and $\|g\|=1$. By the previous part of the theorem  $\mathcal{F}_{\Delta,L}^{\alpha}(g)=g$, and hence $ \|\mathcal{F}_{\Delta,L}^{\alpha}(g)\|_{\infty} = \|g\|_{\infty}.$ The definition of the operator norm now yields $ 1 \le \|\mathcal{F}_{\Delta,L}^{\alpha}\|_{\infty} .$
           \item For $\| \alpha \|_{\infty} < \|L\|^{-1},$  we know that  $\mathcal{F}_{\Delta,L}^{\alpha}:\mathcal{C}(I \times J, \mathbb{R}) \rightarrow \mathcal{C}(I \times J, \mathbb{R})$ is one-one. Note that the range space of $\mathcal{F}_{\Delta,L}^{\alpha}$ is  infinite dimensional. We define the inverse map $\big(\mathcal{F}_{\Delta,L}^{\alpha}\big)^{-1}  : \mathcal{F}_{\Delta,L}^{\alpha}\big(\mathcal{C}(I \times J, \mathbb{R})\big) \rightarrow \mathcal{C}(I \times J, \mathbb{R}).$ With this choice of $\alpha$, $\mathcal{F}_{\Delta,L}^{\alpha}$ is bounded below, and hence it follows that $\big(\mathcal{F}_{\Delta,L}^{\alpha}\big)^{-1}$ is a bounded linear operator. Assume that $\mathcal{F}_{\Delta,L}^{\alpha}$ is a compact operator. Then by  Lemma \ref{BFOel2}, we deduce that the operator $T= \mathcal{F}_{\Delta,L}^{\alpha}\big(\mathcal{F}_{\Delta,L}^{\alpha}\big)^{-1}: \mathcal{F}_{\Delta,L}^{\alpha}\big(\mathcal{C}(I \times J, \mathbb{R})\big) \rightarrow \mathcal{C}(I \times J, \mathbb{R})$ is a compact operator, which is a contradiction to the infinite dimensionality of the space $\mathcal{F}_{\Delta,L}^{\alpha}\big(\mathcal{C}(I \times J,\mathbb{R})\big).$ Therefore, $\mathcal{F}_{\Delta,L}^{\alpha}$ is not a compact operator.
               \item  Under the hypothesis, range space of $\mathcal{F}_{\Delta,L}^{\alpha}$ is closed. Furthermore, $\mathcal{F}_{\Delta,L}^{\alpha}$ is invertible. Recall that if $T: X \to Y$ is invertible, then $T^*$ is also invertible \cite{BB}.  Therefore $ (\mathcal{F}_{\Delta,L}^{\alpha})^*$ is invertible. As a consequence, $\mathcal{F}_{\Delta,L}^{\alpha}$ is Fredholm. The index of a Fredholm operator is defined as $$ \text{index} \big(\mathcal{F}_{\Delta,L}^{\alpha}\big)= \dim \Big(\ker(\mathcal{F}_{\Delta,L}^{\alpha})\Big)- \dim\Big(\ker \big(\mathcal{F}_{\Delta,L}^{\alpha}\big)^*\Big).$$ Hence, the index is zero.\end{enumerate}
 \end{proof}
 \begin{remark}
 If $\alpha(x,y)$ is a non-zero constant function in $I \times J $ and $f$ is a fixed point of $\mathcal{F}_{\Delta,L}^{\alpha}$, then $f$ is a fixed point of $L$ as well. This can be easily seen as follows. Here $\alpha(x,y)= \alpha \ne 0 ,$ a constant function on $I \times J.$ Now let $f$ be a fixed point of $\mathcal{F}_{\Delta,L}^{\alpha},$ that is, $f_{\Delta,L}^{\alpha}=f.$ For $(x,y) \in I \times J$, by the functional equation we have  $$ f\big(u_i(x),v_j(y)\big) =f\big(u_i(x),v_j(y)\big)+ \alpha ~~ \big[f(x,y)- Lf(x,y)\big],$$  from which it follows that $Lf=f.$
 \end{remark}
 \begin{theorem}
 Let $f \in \mathcal{C}(I \times J, \mathbb{R})$.
 \begin{enumerate}

\item If  $\alpha_n \in \mathcal{C}(I \times J, \mathbb{R})$ be such that $\|\alpha_n\|_\infty <1$ and $\alpha_n \to 0$ as $n \to \infty$. Then the corresponding sequence of $\alpha$-fractal functions $f^{\alpha_n}_{\Delta, L} \to f$ as $n \to \infty$.

 \item If $L_n : \mathcal{C}(I \times J, \mathbb{R}) \to \mathcal{C}(I \times J, \mathbb{R})$ be a sequence of bounded linear operators such that $L_nf \neq f$, $(L_nf)(x_i,y_j)= f(x_i, y_j)$ for $(x_i,y_j) \in \partial \Sigma_{N,0} \times \partial \Sigma_{M,0}$ with respect to a net $\Delta$,
     and $L_nf \to f$. Then the corresponding sequence of $\alpha$-fractal functions $f_{\Delta, L_n}^\alpha \to f$ as $n \to \infty$ for any fixed admissible choice of the scale function $\alpha$.
 \end{enumerate}
 \end{theorem}
 \begin{proof}
 From item (1) in the previous theorem, we note that the uniform error bounds for the process of approximation of $f$ with $f_{\Delta, L}^\alpha$ is given by
 $$ \| f^{\alpha}_{\Delta, L} - f \|_{\infty} \leq \frac{\| \alpha \|_{\infty}}{1- \|\alpha\|_{\infty}}~\| f - Lf\|_{\infty}
 .$$
 From this the required results can be deduced.
 \end{proof}
 \begin{remark}
 For an example of a sequence $(L_n)$ satisfying conditions required in the previous theorem, one can work with the two dimensional  Bernstein operators. Let us recall that if $f(x,y)$ is continuous in the square $S=[0,1] \times [0,1]$, then
 $$\lim_{m, n \to \infty} B_{m,n} f(x,y) =f(x,y),$$
uniformly in $x$ and $y$, as $n, m$ approach infinity in any manner whatsoever.  Here
$$B_{n,m}(x,y)= \sum_{i=0}^n \sum_{k=0}^m f\Big(\frac{i}{n}, \frac{k}{m}\Big) p_{i,n}(x)p_{k,m}(y), $$
where $$p_{j,s}(z)= {s \choose j} z^j (1-z)^{s-j}.$$
Similarly, one can work with many extensions of Bernstein operator known in literature.
\end{remark}
A simple reformulation of the above theorem is  given below.
\begin{theorem}
Let $f \in \mathcal{C}(I \times J, \mathbb{R})$ and $\epsilon >0$. Then there exists a bivariate fractal function $f_{\Delta, L}^\alpha$ obtained via the fractal perturbation process given above such that $$\|f- f_{\Delta, L}^ \alpha\| _\infty < \epsilon.$$
\end{theorem}
\begin{proof}
Choose $f_{\Delta, L_n}^\alpha$ or $f_{\Delta, L}^{\alpha_n}$ for a suitably large $n$.
\end{proof}
\begin{definition}
  Given a Banach space $X$ and a bounded linear  operator $T: X \to X$,  a subspace $Y \subseteq X $ is called invariant under $T$ or $T$-invariant if $T(Y) \subseteq Y$. We call, $Y$ is an invariant subspace of $T$.
  \end{definition}
  \begin{remark}
  Clearly, $Y=\{0\}$ and $Y=X$ are $T$-invariant subspaces for every bounded linear operator $T: X \to X$. Consequently, one is interested only in the other invariant subspaces, so-called the non-trivial invariant subspaces.
  \end{remark}
  Next pair of lemmas is fundamental in functional analysis; for instance, these are exercises in \cite{BB}.  We give the proofs here for the sake of completeness.
  \begin{lemma} \label{BFOTHM3}
  Let $X$ be a non-separable Banach space, then every bounded linear operator $T: X \to X$ has a non-trivial closed invariant subspace.
  \end{lemma}
  \begin{proof}
  Suppose $T$ has no non-trivial closed invariant subspace.
  Choose a non-zero element $x$ of $X$ and define $Y_x= \text{span} \big\{x, T(x), T^2(x),...\big\}$. Clearly, $T(Y_x) \subseteq Y_x$. For closedness, we define $ M= \overline{ Y_x}$. It is obvious that $T(M) \subseteq M$, that is $M$ is a closed invariant subspace of $T$. Clearly, $M \neq \{0\}$. If $M=X$, then we obtain a  subset (having rational coefficients in the linear combinations) of $ Y_x$ which is dense in $X$, contradicting  the hypothesis that $X$ is non-separable.
  \end{proof}
\begin{lemma}\label{BFOTHM4}
 If $Y$ is a closed invariant subspace of $T^*$, then $  ^a Y $ is a closed invariant subspace of $T$.
  \end{lemma}
  \begin{proof}
  Let $y \in T(^aY)$, where $y = T(x)$ for an element $x$ in $^aY$. Let $x^*$ be arbitrary in $Y$.
  \begin{equation*}
  \langle y,x^* \rangle = \langle T(x),x^* \rangle
         = \langle x, T^*(x^*) \rangle
         =0.
  \end{equation*}
  In the above, the first two equalities are obvious whereas the last follows from the hypothesis that $Y$ is an invariant subspace under $T^*$. Finally we have, $T(^ a Y) \subseteq ^a Y$, and being the intersection of null spaces of functional from $X^*$, $^a Y$ is closed.
  \end{proof}
\begin{theorem}
   There exists a non-trivial closed invariant subspace for the fractal operator $ \mathcal{F}_{\Delta, L}^{\alpha}: \mathcal{C}(I \times J, \mathbb{R}) \to \mathcal{C}(I \times J, \mathbb{R})$.
   \end{theorem}
   \begin{proof}
   We know that  the adjoint of $\mathcal{F}_{\Delta, L}^{\alpha}$ denoted by $ (\mathcal{F}_{\Delta,L}^{\alpha})^* :\big (\mathcal{C}(I \times J, \mathbb{R})\big)^* \rightarrow \big(\mathcal{C}(I \times J, \mathbb{R})\big)^*$ is a bounded linear operator. Since $\big(\mathcal{C}(I \times J, \mathbb{R})\big)^*$ is a non-separable Banach space, from Lemma \ref{BFOTHM3} it follows that $ (\mathcal{F}_{\Delta,L}^{\alpha})^*$ has a non-trivial closed invariant subspace. Lemma \ref{BFOTHM4} now yields that the fractal operator $ \mathcal{F}_{\Delta,L}^{\alpha}$ has a non-trivial closed invariant subspace.
   \end{proof}
   \begin{remark}
Having established the existence, it is natural to ask for a description of a closed invariant subspace of $\mathcal{F}_{\Delta, L}^\alpha$. This still remains an open question.
\end{remark}
The following remarks are straightforward, however worth recording.
 \begin{remark}
 Assume that $\alpha: I \times J \rightarrow \mathbb{R}$ is a nonzero constant function. Furthermore, assume that the mappings $u_i$ and $v_j$ are affine functions, that is, $ u_i(x)=a_ix+b_i$ and $v_j(y)=c_jy+d_j.$  Let  an invariant subspace $A$ of $\mathcal{F}_{\Delta, L}^{\alpha}$ satisfies the following condition, for any $g(x)=\beta x+ \gamma$ and $h(y)=\delta y+ \kappa$, $f\big(g(.),h(.)\big) \in A$ for all $f \in A.$ Then $A$ is invariant subspace for $L$ as well. To see this, let $f \in A.$ The functional equation can be written in the following form
  $$\alpha  Lf(x,y)=f\big( u_i(x),v_j(y)\big)+\alpha f_{\Delta,L}^{\alpha}(x,y)-f_{\Delta,L}^{\alpha}\big( u_i(x),v_j(y)\big).$$
   Since $\mathcal{F}_{\Delta,L}^{\alpha}(A) \subseteq A$  and $f \in A$, we have $f_{\Delta,L}^{\alpha} \in A.$ The condition on $A$ yields $f \big( u_i(x),v_j(y)\big) \in A$ and $f_{\Delta,L}^{\alpha}\big( u_i(x),v_j(y)\big) \in A,$ since $ u_i(x)=a_ix+b_i$ and $v_j(y)=c_jy+d_j.$ Therefore, the function on the right side of the above equation is in $A,$ that is,  $Lf$ is  in $A.$

 \end{remark}
  \begin{remark}
   Consider $I=J=[0,1]$.
  Let $\mathcal{F}_{\Delta,L}^{\alpha}:\mathcal{C}([0,1]^2, \mathbb{R})\rightarrow \mathcal{C}([0,1]^2, \mathbb{R})$ be the fractal operator corresponding to an admissible scale function $\alpha.$
  Let $L$ be the Bernstein operator, that is, $Lf=B_{k,l}(f)$ and $A$ be the set of polynomials of degree at most $(m,n)$ where $m \ge k$ and $n \ge l.$ Clearly, $A$ is invariant under the Bernstein operator $L.$ Moreover, $A$, being finite dimensional, is a closed subspace of $\mathcal{C}([0,1]^2).$  One can easily check that functions in $A$ also satisfy the condition stated in the above remark. We anticipate that in most cases the self-referential function $f_{\Delta, L}^\alpha$ has noninteger box counting dimension and consequently $f \mapsto f_{\Delta, L}^\alpha$ is a roughing operation; see also Remark \ref{BFOrerem1}.  Therefore the class $\mathcal{F}_{\Delta,L}^{\alpha}(A)$ contains non-smooth functions and $\mathcal{F}_{\Delta,L}^{\alpha}(A) \subseteq A$ does not hold, in general.
  \end{remark}
\section{Extension to $\mathcal{L}^p(I \times J,\mathbb{C})$ and some properties}\label{BFOsecc}
 Here we extend the notion of fractal function to $\mathcal{L}^p(I \times J, \mathbb{C})$ spaces, $1 \le p < \infty$. We demonstrate that corresponding to a given $f \in \mathcal{L}^p(I \times J, \mathbb{C})$,
 there exists a self-referential function $\overline{f}_{\Delta,L}^\alpha \in \mathcal{L}^p(I \times J, \mathbb{C})$.
 As an interlude, in Lemma \ref{BFOLemmaa} below, we define a fractal operator on the space of continuous complex valued functions on $I \times J$, denoted by $\mathcal{C}(I\times J, \mathbb{C})$, endowed with the $\mathcal{L}^p$-norm.
\begin{theorem}\label{BFOThma}
Let $\mathcal{C}(I\times J, \mathbb{R})$ be endowed with the $\mathcal{L}^p$-norm, $1 \le p < \infty$,  and $f \in \mathcal{C}(I \times J,\mathbb{R})$. Further, let $L$ be a linear map bounded with respect to the $\mathcal{L}^p-norm$ on $\mathcal{C}(I \times J,\mathbb{R})$. Then the following inequality holds:
$$ \| f - f_{\Delta,L}^\alpha \|_{\mathcal{L}^p} \le  \frac{\| \alpha \|_{\infty}}{1- \|\alpha\|_{\infty}} \|f-Lf\|_{\mathcal{L}^p}.$$
Consequently, the fractal operator $\mathcal{F}_{\Delta,L}^\alpha$ is bounded.
\end{theorem}
\begin{proof}
Note that
\begin{equation*}
\begin{split}
\| f_{\Delta, L}^\alpha - f\|^p_{\mathcal{L}^p} = &~ \int_{I \times J} | f^\alpha _{\Delta,L} (x,y)- f(x,y)|^p \mathrm{d}x \mathrm{d}y\\
   = &~ \sum_{(i,j)} \int_{I_i \times J_j} | f^\alpha _{\Delta,L} (x,y)- f(x,y)|^p \mathrm{d}x \mathrm{d}y.
\end{split}
\end{equation*}
Using the functional equation for $f_{\Delta, L}^\alpha$ given in (\ref{fnleq}) we obtain
\begin{equation*}
\begin{split}
\| f_{\Delta, L}^\alpha - f\|^p_{\mathcal{L}^p} = &~ \sum_{(i,j)} \int_{I_i \times J_j} \Big | \alpha(x,y) \Big[f_{\Delta,L}^\alpha\big(u_i^{-1}(x), v_j^{-1}(y)\big)-Lf\big(u_i^{-1}(x), v_j^{-1}(y)\big)  \Big]  \Big|^p \mathrm{d}x \mathrm{d}y\\
\le &~ \sum_{(i,j)}\int_{I_i \times J_j} \| \alpha\|_\infty ^p \Big |f_{\Delta,L}^\alpha\big(u_i^{-1}(x), v_j^{-1}(y)\big)-Lf\big(u_i^{-1}(x), v_j^{-1}(y)\big)  \Big|^p \mathrm{d}x \mathrm{d}y.
\end{split}
\end{equation*}
 Changing the variable $(x,y)$ to $(\tilde{x},\tilde{y})$ through the transformation $\tilde{x}=u_i^{-1}(x)$, $\tilde{y}=v_j^{-1}(y)$ and using the change of variable formula for double integrals, we have
\begin{equation*}
\begin{split}
\| f_{\Delta, L}^\alpha - f\|^p_{\mathcal{L}^p} \le \sum_{(i,j)}\int_{I \times J} \| \alpha\|_\infty ^p  \big | f_{\Delta,L}^\alpha (\tilde{x},\tilde{y})-Lf (\tilde{x},\tilde{y}) \big|^p  \Big | \frac{\partial(x,y)} {\partial(\tilde{x},\tilde{y})}  \Big| \mathrm{d}\tilde{x} \mathrm{d} \tilde{y}.
\end{split}
\end{equation*}
Therefore
$$ \| f_{\Delta, L}^\alpha - f\|^p_{\mathcal{L}^p} \le \| \alpha\|_\infty ^p  \| f_{\Delta, L}^\alpha - Lf\|^p_{\mathcal{L}^p} \sum_{(i,j)} |a_i| |c_j|.$$
and hence
$$ \| f_{\Delta, L}^\alpha - f\|_{\mathcal{L}^p} \le \| \alpha\|_\infty   \| f_{\Delta, L}^\alpha - Lf\|_{\mathcal{L}^p}.$$
Using this and the triangle inequality
$$ \| f_{\Delta, L}^\alpha - f\|_{\mathcal{L}^p} \le \| \alpha\|_\infty  \big[ \| f_{\Delta, L}^\alpha - f\|_{\mathcal{L}^p} +\|f-Lf\|_{\mathcal{L}^p}\big], $$
whence $$ \| f - f_{\Delta,L}^\alpha \|_{\mathcal{L}^p} \le  \frac{\| \alpha \|_{\infty}}{1- \|\alpha\|_{\infty}} \|f-Lf\|_{\mathcal{L}^p}.$$
From the above bound for the perturbation error  we have
$$ \| f_{\Delta,L}^\alpha\|_{\mathcal{L}^p} - \|f\|_{\mathcal{L}^p} \le \| f - f_{\Delta,L}^\alpha \|_{\mathcal{L}^p} \le  \frac{\| \alpha \|_{\infty}}{1- \|\alpha\|_{\infty}} \|f-Lf\|_{\mathcal{L}^p},$$
using which we infer that
$$ \|\mathcal{F}_{\Delta,L}^\alpha (f) \|_{\mathcal{L}^p} \le  \Big[1+ \frac{\| \alpha \|_{\infty}}{1- \|\alpha\|_{\infty}}  \|Id-L\| \Big] \|f\|_{\mathcal{L}^p}.$$
That is, $\mathcal{F}^\alpha_{\Delta,L}$ is a bounded operator.
\end{proof}
 In what follows, for the notational convenience, we may suppress the dependence on $\Delta$, $L$ to  denote the bivariate $\alpha$-fractal function corresponding to $f$ by $f^\alpha$ and the fractal operator by $\mathcal{F}^{\alpha}$.
 \begin{lemma} \label{BFOLemmaa}
 Let $\mathcal{F}^{\alpha}$ be a fractal operator on $\mathcal{C}(I\times J, \mathbb{R})$, endowed with the $\mathcal{L}^p$-norm. The operator  $\mathcal{F}^{\alpha}_{\mathbb{C}}: \mathcal{C}(I\times J, \mathbb{C}) \rightarrow \mathcal{C}(I\times J, \mathbb{C})$ defined by $$\mathcal{F}^{\alpha}_{\mathbb{C}}(f)= \mathcal{F}^{\alpha}_{\mathbb{C}}(f_1+ i f_2) =  \mathcal{F}^{\alpha}(f_1)+ i~  \mathcal{F}^{\alpha}(f_2)$$  is  a bounded linear operator.
 \end{lemma}
 \begin{proof}
 Since $\mathcal{F}^{\alpha}$ is linear, $\mathcal{F}^{\alpha}_{\mathbb{C}}$ is linear. It remains to show that $\mathcal{F}^{\alpha}_{\mathbb{C}}$ is a bounded operator. For this,
 \begin{equation} \label{BFOreveq5}
 \begin{split}
 \|\mathcal{F}^{\alpha}_{\mathbb{C}}(f)\|_{\mathcal{L}^p}^p= &~ \int_{I \times J}|\mathcal{F}_{\mathbb{C}}^{\alpha}(f)|^p \mathrm{d}x\mathrm{d}y\\
 =&~\int_{I \times J} \Big[ |\mathcal{F}^{\alpha}(f_1)|^2+ |\mathcal{F}^{\alpha}(f_2)|^2 \Big]^{\frac{p}{2}} \mathrm{d}x\mathrm{d}y\\
 \le &~ 2^{\frac{p}{2}} \int_{I \times J} \Big [|\mathcal{F}^{\alpha}(f_1)|^p+ |\mathcal{F}^{\alpha}(f_2)|^p \Big] \mathrm{d}x\mathrm{d}y\\
 =&~ 2^{\frac{p}{2}} \Big[ \|\mathcal{F}^{\alpha}(f_1) \|_{\mathcal{L}^p}^p + \| \mathcal{F}^{\alpha}(f_2) \|_{\mathcal{L}^p}^p\Big]\\
 \le &~ 2^{\frac{p}{2}}\|\mathcal{F}^{\alpha}\|^p \Big[ \|f_1\|_{\mathcal{L}^p}^p + \|f_2\|_{\mathcal{L}^p}^p   \Big]\\
 =&~2^{\frac{p}{2}}\|\mathcal{F}^{\alpha}\|^p \Big[\int_{I \times J} |f_1|^p \mathrm{d}x\mathrm{d}y +  \int_{I \times J} |f_2|^p \mathrm{d}x\mathrm{d}y
    \Big]\\
    \le &~ 2^{\frac{p}{2}+1}\|\mathcal{F}^{\alpha}\|^p \int_{I \times J} |f|^p \mathrm{d}x\mathrm{d}y \\
    = &~ 2^{\frac{p}{2}+1}\|\mathcal{F}^{\alpha}\|^p \|f\|_{\mathcal{L}^p}^p,
 \end{split}
 \end{equation}
thus we have  $$\|\mathcal{F}^{\alpha}_{\mathbb{C}}(f)\|_{\mathcal{L}^p} \le  2^{\frac{1}{2}+\frac{1}{p}}\|\mathcal{F}^{\alpha}\| \|f\|_{\mathcal{L}^p},$$
  proving that $\mathcal{F}^{\alpha}_{\mathbb{C}}$ is a bounded operator and $\|\mathcal{F}^{\alpha}_{\mathbb{C}} \| \le  2^{\frac{1}{2}+\frac{1}{p}} \|\mathcal{F}^{\alpha}\|.$
\end{proof}
\begin{remark}\label{BFOrevrem5}
The case $p=2$ is better behaved in the following sense.  If $p=2$, then in place of (\ref{BFOreveq5}) we have
 \begin{equation*}
 \begin{split}
 \|\mathcal{F}^{\alpha}_{\mathbb{C}}(f)\|_{\mathcal{L}^2}^2= &~\int_{I \times J}|\mathcal{F}_{\mathbb{C}}^{\alpha}(f)|^2 \mathrm{d}x\mathrm{d}y\\
  =&~\int_{I \times J} \Big[ |\mathcal{F}^{\alpha}(f_1)|^2+ |\mathcal{F}^{\alpha}(f_2)|^2 \Big] \mathrm{d}x\mathrm{d}y\\
  =&~ \| \mathcal{F}^{\alpha}(f_1) \|_{\mathcal{L}^2}^2 + \| \mathcal{F}^{\alpha}(f_2) \|_{\mathcal{L}^2}^2 \\
 \le &~ \| \mathcal{F}^{\alpha}\|^2  \| f_1 + i f_2 \|_{\mathcal{L}^2}^2 \\
 =&~ \| \mathcal{F}^{\alpha}\|^2 \| f\|_{\mathcal{L}^2}^2.
\end{split}
 \end{equation*}
Consequently, $ \|\mathcal{F}^{\alpha}_{\mathbb{C}}\| \le \| \mathcal{F}^{\alpha}\|$, improving the bound obtained in the previous lemma.
\end{remark}
 Let us remind the following fundamental theorem.
 \begin{theorem} \cite{Limaye}
 If an operator $T:Z \rightarrow Y$ is linear and bounded, $Y$ is a Banach space and $Z$ is dense in $X,$ then $T$ can be extended to $X$ preserving the norm of $T.$
 \end{theorem}
 In the sequel, we fix the notation $M:= 2^{\frac{1}{2}+\frac{1}{p}}$.
 \begin{theorem}
 Let $1 \le p < \infty$, $\mathcal{C}(I \times J, \mathbb{R})$ be equipped with $\mathcal{L}_p$-norm and
 $\mathcal{F}_{\mathbb{C}}^{\alpha}: \mathcal{C}(I \times J, \mathbb{C}) \rightarrow \mathcal{C}(I \times J,\mathbb{C})$ be the fractal operator defined in
 the previous lemma. Then there exists a bounded linear operator $\overline{\mathcal{F}}^{\alpha}_{\mathbb{C}}: \mathcal{L}^p(I\times J,\mathbb{C}) \rightarrow \mathcal{L}^p(I \times J,\mathbb{C})$ such that its restriction to  $\mathcal{C}(I \times J, \mathbb{C})$ is $\mathcal{F}_{\mathbb{C}}^{\alpha}$ and $\| \mathcal{F}^{\alpha} \| \le\|\overline{\mathcal{F}}_{\mathbb{C}}^{\alpha} \| = \| \mathcal{F}_{\mathbb{C}}^{\alpha} \| \le M \| \mathcal{F}^\alpha\|$.
 \end{theorem}
 \begin{proof}
 It is well-known that  $\mathcal{C}(I \times J, \mathbb{C})$ is dense in  $\mathcal{L}^p(I \times J, \mathbb{C}),$ for $1 \le p < \infty$. From the previous lemma, we have
   a bounded linear operator $\mathcal{F}^{\alpha}_{\mathbb{C}}: \mathcal{C}(I\times J, \mathbb{C}) \rightarrow \mathcal{C}(I\times J, \mathbb{C}).$ Now using the previous theorem we conclude that there exists a bounded linear operator $\overline{\mathcal{F}}_{\mathbb{C}}^{\alpha}: \mathcal{L}^p(I\times J,\mathbb{C}) \rightarrow \mathcal{L}^p(I \times J,\mathbb{C})$ such that $$\|\overline{\mathcal{F}}_{\mathbb{C}}^{\alpha} \| = \|\mathcal{F}^{\alpha}_{ \mathbb{C}} \| \le M \| \mathcal{F}^{\alpha} \|.$$ Furthermore, the previous theorem gives $\overline{\mathcal{F}}_{ \mathbb{C}}^{\alpha}(f) = \mathcal{F}^{\alpha}_{ \mathbb{C}} (f)=  \mathcal{F}^{\alpha} (f),~~ \forall~ f \in \mathcal{C}(I \times J, \mathbb{R}).$ Hence
   \begin{equation*}
   \begin{split}
   \| \mathcal{F}^{\alpha} \| :=&~ \sup \big\{ \|\mathcal{F}^{\alpha}(f) \|_{\mathcal{L}^p} : f \in \mathcal{C}(I \times J, \mathbb{R}),\|f\|_{\mathcal{L}^p} =1 \big\}\\ =&~ \sup \big\{\|\overline{\mathcal{F}}_{\mathbb{C}}^{\alpha} (f)\|_{\mathcal{L}^p}: f \in C(I \times J, \mathbb{R}), \|f\|_{\mathcal{L}^p} =1  \big\}.
   \end{split}
   \end{equation*}
 This implies that $ \| \mathcal{F}^{\alpha} \| \le \|\overline{\mathcal{F}}_{\mathbb{C}}^{\alpha} \| $ and hence that $ \| \mathcal{F}^{\alpha} \| \le \|\overline{\mathcal{F}}_{\mathbb{C}}^{\alpha} \| \le M \| \mathcal{F}^{\alpha} \|. $
 \end{proof}
 \begin{remark}
 In view of Remark \ref{BFOrevrem5}, for $p=2$, we have $\|\overline{\mathcal{F}}_{\mathbb{C}}^{\alpha} \|=\| \mathcal{F}^{\alpha} \|$; for univariate counterpart see \cite{M1}.
 \end{remark}
 \begin{lemma}
 Let $L: \mathcal{C}(I \times J, \mathbb{R}) \rightarrow \mathcal{C}(I \times J, \mathbb{R})$ be a bounded linear operator with respect to the $\mathcal{L}^p$-norm on $\mathcal{C}(I \times J, \mathbb{R})$, $1 \le p< \infty$.  Then there exists a bounded linear operator $\overline{L}_{\mathbb{C}}: \mathcal{L}^p(I \times J, \mathbb{C}) \rightarrow \mathcal{L}^p(I \times J, \mathbb{C})$ such that  the restriction of $\overline{L}_{\mathbb{C}}$ to $\mathcal{C}(I \times J, \mathbb{R})$ is  $L$ and $ \|L\| \le \|\overline{L}_{\mathbb{C}}\|\le M \|L\|.$ Moreover, the extension of $Id-L$ is $Id - \overline{L}_{\mathbb{C}} $.
 \end{lemma}
 \begin{proof}
 The first part of the lemma can be proved similar to Lemma \ref{BFOLemmaa} via $L_{\mathbb{C}}.$ For the rest, let $f \in  \mathcal{L}^p(I \times J, \mathbb{C})$ and a sequence of complex-valued continuous functions such that $ f_n \rightarrow f$ with respect to the $\mathcal{L}^p$-norm. Then $$ (Id-L)_{\mathbb{C}}f_n = f_n -L_{\mathbb{C}}f_n \rightarrow f- \overline{L}_{\mathbb{C}}f $$ and hence $\overline{(Id - L)_{\mathbb{C}}}=Id - \overline{L}_{\mathbb{C}}.$
\end{proof}
 \begin{lemma}
  For each $f \in \mathcal{C}(I \times J, \mathbb{C}),$ $$\|\mathcal{F}^{\alpha}_{\mathbb{C}}(f)-f\|_{\mathcal{L}^p} \le M \|\alpha\|_{\infty}\|\mathcal{F}^{\alpha}_{\mathbb{C}}(f) - L_{\mathbb{C}}(f)\|_{\mathcal{L}^p}.$$
  \end{lemma}
  \begin{proof}
  Let $f \in \mathcal{C}(I \times J, \mathbb{C}),$ where $f=f_1+ i f_2$ for some $f_1,f_2 \in \mathcal{C}(I \times J, \mathbb{R}).$ Since $\mathcal{F}^{\alpha}_{\mathbb{C}}(f):=\mathcal{F}^{\alpha}(f_1)+ i \mathcal{F}^{\alpha}(f_2),$ we have $$ Re\big(\mathcal{F}^{\alpha}_{\mathbb{C}}(f)\big)=\mathcal{F}^{\alpha}(f_1), \quad Im\big(\mathcal{F}^{\alpha}_{\mathbb{C}}(f)\big)=\mathcal{F}^{\alpha}(f_2).$$ Furthermore, Using Theorem \ref{BFOThma} and some basic inequalities, we have  \begin{equation*}
   \begin{split}
   \|\mathcal{F}^{\alpha}_{\mathbb{C}}(f)- f \|_{\mathcal{L}^p}^p= &~ \int_{I \times J}|\mathcal{F}_{\mathbb{C}}^{\alpha}(f)- f|^p \mathrm{d}x\mathrm{d}y\\
   =&~\int_{I \times J} \Big[ |\mathcal{F}^{\alpha}(f_1)-f_1|^2+ |\mathcal{F}^{\alpha}(f_2)-f_2|^2 \Big]^{\frac{p}{2}} \mathrm{d}x\mathrm{d}y\\
   \le &~ 2^{\frac{p}{2}} \int_{I \times J} \Big [|\mathcal{F}^{\alpha}(f_1)-f_1|^p+ |\mathcal{F}^{\alpha}(f_2)-f_2|^p \Big] \mathrm{d}x\mathrm{d}y\\
   =&~ 2^{\frac{p}{2}} \Big[ \|\mathcal{F}^{\alpha}(f_1)-f_1 \|_{\mathcal{L}^p}^p + \| \mathcal{F}^{\alpha}(f_2)- f_2 \|_{\mathcal{L}^p}^p\Big]\\
   \le &~ 2^{\frac{p}{2}}\|\alpha\|_{\infty}^p \Big[ \|\mathcal{F}^{\alpha}(f_1)- L(f_1)\|_{\mathcal{L}^p}^p + \|\mathcal{F}^{\alpha}(f_2)- L(f_2)\|_{\mathcal{L}^p}^p   \Big]\\
   =&~2^{\frac{p}{2}}\|{\alpha}\|_{\infty}^p \Big[\int_{I \times J} |\mathcal{F}^{\alpha}(f_1)- L(f_1)|^p \mathrm{d}x\mathrm{d}y +  \int_{I \times J} |\mathcal{F}^{\alpha}(f_2)- L(f_2)|^p \mathrm{d}x\mathrm{d}y
      \Big]\\
      \le &~ 2^{\frac{p}{2}+1}\|{\alpha}\|_{\infty}^p \int_{I \times J} |\mathcal{F}_{\mathbb{C}}^{\alpha}(f)- L_{\mathbb{C}}(f)|^p \mathrm{d}x\mathrm{d}y \\
      = &~ 2^{\frac{p}{2}+1}\|{\alpha}\|_{\infty}^p \|\mathcal{F}_{\mathbb{C}}^{\alpha}(f)- L_{\mathbb{C}}(f)\|_{\mathcal{L}^p}^p,
   \end{split}
   \end{equation*}
   hence the proof.
\end{proof}
  The proof of the next corollary follows from the way in which $\overline{\mathcal{F}}_{\mathbb{C}}^{\alpha}$ is defined using denseness of $\mathcal{C}(I \times J, \mathbb{C})$ in $\mathcal{L}^p(I \times J, \mathbb{C})$.
 \begin{corollary}
 For each $f \in \mathcal{L}^p(I \times J, \mathbb{C}),$ $$\|\overline{\mathcal{F}}_{\mathbb{C}}^{\alpha}(f)-f\|_{\mathcal{L}^p} \le M \|\alpha\|_{\infty}\|\overline{\mathcal{F}}_{\mathbb{C}}^{\alpha}(f) - \overline{L}_{\mathbb{C}}f\|_{\mathcal{L}^p}.$$
 \end{corollary}
 Using the previous lemma, we deduce the next theorem in a similar way as that in Theorem \ref{BFOTHM2}. We avoid the proof, however recall that for a bounded linear operator $T:X \to X$ in a Hilbert space $X$, the following orthogonal decomposition holds:

 $$X = \overline{\text{Range} (T)} \oplus \text{Ker} (T^*).$$
 \begin{theorem}
  Let $\| \alpha\|_{\infty} = \sup \big\{|\alpha(x,y)|: (x,y) \in I \times J \big\} $, $M \| \alpha \|_{\infty} < 1$ and let $Id$ be the identity operator on $ \mathcal{L}^p(I \times J, \mathbb{C})$.
  \begin{enumerate}
  \item For any $f \in \mathcal{L}^p(I \times J, \mathbb{C})$, the perturbation error satisfies $$\| \overline{\mathcal{F}}_{\mathbb{C}}^{\alpha} (f)- f\|_{\mathcal{L}^p} \leq \frac{M^2 \|\alpha\|_{\infty} \| Id - L\|_{\mathcal{L}^p}}{1-M \|\alpha\|_{\infty}}\|f\|_{\mathcal{L}^p}.$$ In particular, If $\| \alpha\|_{\infty}=0$ then $\overline{\mathcal{F}}_{\mathbb{C}}^{\alpha}=Id.$
  \item   The  norm  of the fractal operator $\overline{\mathcal{F}}_{\mathbb{C}}^{\alpha}: \mathcal{L}^p(I \times J, \mathbb{C}) \to \mathcal{L}^p(I \times J, \mathbb{C})$ satisfies  $$\| \overline{\mathcal{F}}_{\mathbb{C}}^{\alpha}\|_{\mathcal{L}^p} \leq 1+ \frac{M^2 \|\alpha\|_{\infty}~~\| Id - L\|}{1- M \| \alpha\|_{\infty}}.$$
  \item  For $\| \alpha\|_{\infty} < M^{-2}\|L\|^{-1}$, $\overline{\mathcal{F}}_{\mathbb{C}}^{\alpha}$ is bounded below. In particular, $\overline{\mathcal{F}}_{\mathbb{C}}^{\alpha}$ is one to one and the range of $\overline{\mathcal{F}}_{\mathbb{C}}^{\alpha}$ is closed.
  \item  $\| \alpha\|_{\infty} < (M + M^2 \|Id - L\|)^{-1}$, then $\overline{\mathcal{F}}_{\mathbb{C}}^{\alpha}$ has a bounded inverse. Moreover, $$\| \big(\overline{\mathcal{F}}_{\mathbb{C}}^{\alpha}\big)^{-1}\| \leq \frac{1+M \| \alpha\|_{\infty} }{1-M^2 \| \alpha\|_{\infty} \|L\|} .$$
  \item  $\| \alpha\|_{\infty} < (M + M^2 \|Id - L\|)^{-1},$ $ \overline{\mathcal{F}}_{\mathbb{C}}^{\alpha}$ is Fredholm and its index is $0.$
  \item  Let $p=2$ and  $\| \alpha\|_{\infty} < \|L\|^{-1}$, then $\mathcal{L}^2(I \times J, \mathbb{C})= \text{Range}\big(\overline{\mathcal{F}}_{\mathbb{C}}^{\alpha}\big) \oplus \text{Ker}\big((\overline{\mathcal{F}}_{\mathbb{C}}^{\alpha})^*\big)$ where $(\overline{\mathcal{F}}_{\mathbb{C}}^{\alpha})^*$ is the adjoint operator of $\overline{\mathcal{F}}_{\mathbb{C}}^{\alpha}.$ \end{enumerate}
\end{theorem}
 \section{Some approximation aspects} \label{BFOsecd}
 In this section we shall return to the bivariate $\alpha$-fractal functions in the function space $\mathcal{C}(I \times J, \mathbb{R})$. First let us recall the following well-known definition.
 \begin{definition}
 A Schauder basis in an infinite dimensional Banach space $X$ is a sequence $(e_n)$ of elements in $X$ satisfying the following condition: for every $x$ in $X$, there is a unique sequence $(a_n(x))$ of scalars such that
 $$x= \sum_{n=1}^\infty a_n(x) e_n, \quad \text{i.e.,} \quad \Big\|x- \sum_{n=1}^m a_n(x) e_n\Big\| \to 0~~ \text{as}~~ m \to \infty.$$
 The coefficients $a_n(x)$ are linear  functions of $x$ uniquely determined by the basis referred to as the associated sequence of coefficient functionals.
 \end{definition}
 The existence of Schauder bases has many practical applications, for instance,  for finding the best approximation of an element in the space, if it exists. Schauder bases are especially important for applications in operator equations in Banach spaces. In previous section, we studied bivariate fractal functions that are close to the prescribed function, at the same time possessing a self-referential structure.  In some applications, it is required to maintain the global structure involved in a given problem and self-referentiality may be beneficial. For simplicity, let us take $I=J=[0,1]$. In contrast to the case $\mathcal{C}([0,1])$ wherein classical Faber-Schauder system provides a Schauder basis, the situation gets more complicated in the case $\mathcal{C}([0,1]^d)$, $d \ge2$. The tensor products of Faber-Schauder bases in the copies of $\mathcal{C}([0,1])$ form a basis of $\mathcal{C}([0,1]^d)$. Another different basis is the so-called regular pyramidal and squew pyramidal bases. The reader may refer \cite{ZS} for a detailed description of various Schauder bases for $\mathcal{C}([0,1]^d)$. In this instance, we find a Schauder basis for $\mathcal{C}(I \times J, \mathbb{R})$ consisting of fractal functions; the maps involved are perturbations of those belonging to a classical basis for $\mathcal{C}(I \times J, \mathbb{R})$. The central idea is to use the fact that a topological automorphism preserves a Schauder basis, however we provide the details in the following.
\begin{theorem}
There exists a Schauder basis consisting of bivariate fractal functions for the space $\mathcal{C}([0,1]^2, \mathbb{R})$ .
  \end{theorem}
\begin{proof}
Let $(e_n)$ be a Schauder basis of $\mathcal{C}([0,1]^2, \mathbb{R})$, whose existence is hinted at the last paragraph. Choose $\alpha$ such that  $\| \alpha \|_{\infty} < (1 + \|Id - L\|)^{-1}$, so that by Theorem \ref{BFOTHM2}, the fractal operator $\mathcal{F}_{\Delta, L}^\alpha$ is a topological automorphism. If $g \in \mathcal{C}([0,1]^2, \mathbb{R})$ then $\big(\mathcal{F}_{\Delta, L}^\alpha\big)^{-1} (g) \in \mathcal{C}([0,1]^2, \mathbb{R})$ , so that
$$ \big(\mathcal{F}_{\Delta, L}^\alpha\big)^{-1} (g) = \sum_{n=1}^\infty a_n \Big( \big(\mathcal{F}_{\Delta, L}^\alpha\big)^{-1} (g) \Big) e_n.$$
By the continuity of the fractal linear operator $\mathcal{F}_{\Delta, L}^\alpha$ it follows that
$$g= \mathcal{F}_{\Delta, L}^\alpha \big(\mathcal{F}_{\Delta, L}^\alpha\big)^{-1} (g)=\sum_{n=1}^\infty a_n \Big( \big(\mathcal{F}_{\Delta, L}^\alpha\big)^{-1} (g) \Big) {e_n}^\alpha,$$
where $e_n^\alpha = \mathcal{F}_{\Delta, L}^\alpha (e_n).$
Assume that $g= \sum_{n=1}^\infty b_n {e_n}^\alpha$ was another representation of $g$. Since $\big(\mathcal{F}_{\Delta, L}^\alpha\big)^{-1}$ is also continuous, we have $$\big(\mathcal{F}_{\Delta, L}^\alpha\big)^{-1} (g) = \sum_{n=1}^\infty b_n e_n $$ and hence that $b_n = \sum_{n=1}^\infty a_n \Big( \big(\mathcal{F}_{\Delta, L}^\alpha\big)^{-1} (g) \Big)$ for each $n$. Consequently, $(e_n^\alpha)$ is a Schauder basis for $\mathcal{C}([0,1]^2, \mathbb{R})$, obtaining the desired conclusion.
\end{proof}

\begin{definition}
Consider the fractal operator $\mathcal{F}_{\Delta,L}^\alpha:\mathcal{C}(I \times J, \mathbb{R})\rightarrow \mathcal{C}(I \times J, \mathbb{R})$ defined  by $f \mapsto f_{\Delta,L}^\alpha$; see Section \ref{BFOsecb}.  Let $p \in \mathcal{C} (I \times J, \mathbb{R})$ be a bivariate polynomial. Then  $\mathcal{F}_{\Delta,L}^\alpha(p)=p_{\Delta,L}^\alpha$, denoted for simplicity by $p^\alpha$, is referred to as a bivariate fractal polynomial; see also \cite{M2}. Let $\mathcal{P}(I \times J)\subset \mathcal{C}(I \times J, \mathbb{R})$ be the space of all bivariate polynomials, then we denote by $\mathcal{P}^\alpha(I \times J)$, the image space $\mathcal{F}^\alpha_{\Delta,L}\big(\mathcal{P}(I \times J)\big)$.
\end{definition}

\begin{notation}
Let $\mathcal{P}_{m,n}(I \times J)$ be the set of all bivariate polynomials of total degree at most $m+n$ defined on $I \times J$. That is,
$$\mathcal{P}_{m,n}(I \times J)=\Big\{p(x,y)= \sum_{i=0}^{m} \sum_{j=0}^{n} a_{ij}x^i y^j: a_{ij} \in \mathbb{R}, 0 \le i \le m ~~and ~~0 \le j \le n \Big\}.$$
We let $\mathcal{P}_{m,n}^{\alpha} (I \times J)= \mathcal{F}_{\Delta,L}^{\alpha}(\mathcal{P}_{m,n}\big(I \times J)\big)$.
  \end{notation}
\begin{remark}
In fact, given an approximation class $X \subset \mathcal{C}(I \times J, \mathbb{R})$, one can obtain a new class of functions by considering the fractal perturbation of functions in $X$, that is,  by considering $\mathcal{F}_{\Delta,L}^\alpha(X)$. For some reasons, perhaps the physical situation which the approximant is intended to
model, finding an irregular approximant with a specified roughness (quantified in terms of the box counting dimension) from a  subset of $\mathcal{C}(I \times J, \mathbb{R})$ to  a given $f \in   \mathcal{C}(I \times J, \mathbb{R})$ is of interest, and one may tackle it with the perturbed approximation class $\mathcal{F}_{\Delta,L}^\alpha(X)$.
\end{remark}
\begin{lemma}\label{BFOlemmae}
For any admissible choice of $\alpha \in \mathcal{C}(I \times J, \mathbb{R})$, the space $\mathcal{P}_{m,n}^{\alpha} (I \times J)$ is finite dimensional. For $\|\alpha\|_\infty < \|L\|^{-1}$, the dimension of $\mathcal{P}_{m,n}^{\alpha} (I \times J)$ is $\frac{(m+n+2)(m+n+1)}{2}$.
\end{lemma}
\begin{proof}
It is well-known that the dimension of $\mathcal{P}_{m,n}(I \times J)$ is $m+n+2 \choose 2$. However, for the sake of exposition let us provide an abbreviated argument here. By a straight forward counting argument,  we see that there are $k+2-1 \choose k$ ways in which $k$ indistinguishable exponents can be distributed to $2$ distinguishable variables. Therefore it follows that the dimension of $\mathcal{P}_{m,n}(I \times J)$ is $\sum_{k=0}^{m+n} {k+2-1 \choose k} = {m+n+2 \choose 2}$:=$r$. Let $\{p_1,p_2,\dots,p_r\}$ be a basis for $\mathcal{P}_{m,n}(I \times J)$. Since $\mathcal{F}_{\Delta,L}^\alpha$ is a linear map, it follows that $\mathcal{P}_{m,n}^{\alpha}$ is spanned by $\{ p_1^\alpha,p_2 ^\alpha,\dots,p_r^\alpha \}$. If  $\|\alpha\|_\infty < \|L\|^{-1}$, then $\mathcal{F}_{\Delta,L}^\alpha$ is injective and hence $\{ p_1^\alpha,p_2 ^\alpha,\dots,p_r^\alpha \}$ is basis for $\mathcal{P}_{m,n}^{\alpha} (I \times J)$, completing the proof.
\end{proof}

Let us recall some basic concepts and a result from approximation theory; see, for instance, \cite{EWC}.

  \begin{definition}
  Let $(X,\|.\|)$ be a normed linear space over $\mathbb{K},$ the field of real or complex numbers. Given a nonempty set $V \subseteq X$ and an element $x \in X ,$ distance from  $x$ to $V$ is defined as $d(x,V)=\inf\{\|x-v\| : v \in V\}.$ If there exists an element $v^*(x) \in V$ such that $\|x- v^*\|= d(x,V),$  we call $v^*$ a best approximant to $x$ from $V.$ A subset $V$ of $X$ is called proximinal (proximal or existence set) if for each $x \in X$ a best approximant $v^*(x) \in V$ of $x$ exists.
  \end{definition}

  \begin{theorem}
  If $V$ is a finite dimensional subspace of the normed linear space $X,$ then for each $x \in X,$ there is a best approximant from $V.$
  \end{theorem}
  The following theorem is a direct consequence of the previous theorem and Lemma \ref{BFOlemmae}.
  \begin{theorem}
  Let $\mathcal{C}(I \times J, \mathbb{R})$ be endowed with the uniform norm. For each $f \in \mathcal{C}(I \times J, \mathbb{R})$, a best approximant $p^\alpha_f $ in $\mathcal{P}_{m,n}^{\alpha}(I \times J)$ exists.
  \end{theorem}
 \begin{theorem} \label{BFOTHM5}
 Let $ \mathcal{C}(I \times J, \mathbb{R})$ be endowed with the uniform norm, $f \in \mathcal{C}(I \times J, \mathbb{R}),$ and $L:\mathcal{C}(I \times J, \mathbb{R})\rightarrow \mathcal{C}(I \times J, \mathbb{R}),$ $L \ne Id$ be a bounded linear operator satisfying $(Lf)(x_i,y_j)=f(x_i,y_j), ~~~~\forall~~ (i,j) \in \partial \Sigma_{N,0} \times \partial \Sigma_{M,0}.$
 For any $\epsilon>0,$ net $\Delta$ of the rectangle $I \times J ,$ there exists a bivariate fractal polynomial $p^{\alpha}$  such that
 $$ \|f- p^{\alpha}\|_{\infty} <\epsilon .$$
 \end{theorem}
 \begin{proof}
 Let $\epsilon >0$ be given. By the Stone-Weierstrass theorem, there exists a polynomial function $ p$ in two variables such that $$ \|f- p\|_{\infty} <\frac{\epsilon}{2} .$$ Fix a net $\Delta$ of the rectangle $I \times J ,$  a bounded linear operator $L:\mathcal{C}(I \times J, \mathbb{R})\rightarrow \mathcal{C}(I \times J, \mathbb{R}),$ $L \ne Id$ satisfying $(Lf)(x_i,y_j)=f(x_i,y_j), ~~~~\forall~~ (i,j) \in \partial \Sigma_{N,0} \times \partial \Sigma_{M,0}.$ Choose $\alpha: I\times J \rightarrow \mathbb{R}$ as continuous function on $I \times J$ with $\|\alpha\|_{\infty}= \sup\big\{|\alpha(x,y)|:(x,y) \in I \times J \big\} < 1$ such that $$ \|\alpha\|_{\infty} <\frac{\frac{\epsilon}{2}}{\frac{\epsilon}{2}+\|Id-L\|~~ \|p\|_{\infty}} .$$
 Then we have
 \begin{equation*}
   \begin{aligned}
    \|f- p^{\alpha}\|_{\infty} & \le \|f- p\|_{\infty}+\|p- p^{\alpha}\|_{\infty}.\\
    & \le \|f- p\|_{\infty}+ \frac{\|\alpha\|_{\infty}}{1-\|\alpha\|_{\infty}}\|Id- L\| ~~\|p\|_{\infty}.\\
    & < \frac{\epsilon}{2} + \frac{\epsilon}{2}.\\
     & = \epsilon.
   \end{aligned}
   \end{equation*}
     In the above, the first inequality is just the triangle inequality, second follows from Theorem \ref{BFOTHM2} and third is obvious.
 \end{proof}
 \begin{remark}
 In the above proof, we selected  $\alpha \in \mathcal{C}(I \times J, \mathbb{R})$, for instance, constants,  such that $\|\alpha\|_{\infty} <\dfrac{\frac{\epsilon}{2}}{\frac{\epsilon}{2}+\|Id-L\|~~ \|p\|_{\infty}} .$ In this case, $\alpha$ may be ``close" to $0$ and hence $p^\alpha$ may lose self-referentiality and behave as a traditional bivariate polynomial. Alternatively, one can fix $\alpha \in \mathcal{C}(I \times J, \mathbb{R})$ such that $\|\alpha\|_{\infty} < 1$, but otherwise arbitrary and choose
a bounded linear operator $L:\mathcal{C}(I \times J, \mathbb{R})\rightarrow \mathcal{C}(I \times J, \mathbb{R}),$ $L \ne Id$ satisfying $(Lf)(x_i,y_j)=f(x_i,y_j), ~~~~\forall~~ (i,j) \in \partial \Sigma_{N,0} \times \partial \Sigma_{M,0}$ such that $$ \|Id-L\|  <\frac{1-\|\alpha\|_{\infty}}{\|\alpha\|_{\infty} \|p\|_{\infty}} \frac{\epsilon}{2} .$$
In this case, we expect that the graph of the corresponding fractal polynomial $p^\alpha$ has the box dimension greater than $2$, thus  possesses a ``fractality" in it and differs from the traditional bivariate polynomial.
 \end{remark}
In view of the previous theorem, we have
\begin{theorem}
 The set of bivariate fractal polynomials with non-null scale vector is dense in $\mathcal{C}(I \times J, \mathbb{R}).$
 \end{theorem}
 In the next theorem,  we provide the denseness of a class of bivariate fractal polynomials which is a proper subset of the dense set considered above. This theorem reveals that one single scale vector is sufficient to obtain a bivariate fractal polynomial
approximation of any bivariate  continuous function.
 \begin{theorem}
 If $\|\alpha\|_{\infty}< (1+\|Id-L\|)^{-1},$ then $\mathcal{P}^{\alpha}(I \times J)$ is dense in $\mathcal{C}(I \times J, \mathbb{R}).$
 \end{theorem}
 \begin{proof}
  Note that  under the given condition on $\alpha,$ $\mathcal{F}_{\Delta,L}^{\alpha}$ is a topological automorphism . Let $f \in \mathcal{C}(I \times J, \mathbb{R})$. By the Stone-Weierstrass theorem, there exists a sequence of bivariate polynomials $(p_n)$ such that $p_n \to (\mathcal{F}_{\Delta,L}^\alpha)^{-1}(f)$ in the uniform norm. Now since $\mathcal{F}_{\Delta,L}^\alpha$ is bounded, we obtain $p_n^\alpha:=\mathcal{F}_{\Delta,L}^\alpha(p_n) \to f$ as $n \to \infty$, and with it the proof.
 \end{proof}
\begin{definition} \cite{Gal}
Let $\mathcal{C}( [-1,1]^2,\mathbb{R})$ be supplied with the uniform norm and $f \in \mathcal{C}( [-1,1]^2,\mathbb{R})$. The Ditzian-Totik modulus of smoothness is defined as
$$ \omega_r^{\phi} (f,\delta_1, \delta_2)= \sup_{0<h_i \le \delta_i, i=1,2} \big|\overline{\Delta}_{h_1 \phi(x), h_2 \phi(y)}^r  f (x,y) \big|,$$
where $\phi(x)= \sqrt{1-x^2}$ and $r$-th symmetric difference of the function $f$ is given by
$$ \overline{\Delta}_{h_1 \phi(x),h_2 \phi(y)}^r f(x,y)= \sum_{k=0}^r (-1)^{k} {r \choose k} f\Big(x+h_1\phi(x) (\frac{r}{2}-k),y + h_2\phi(y) (\frac{r}{2}-k) \Big)$$
if $\big(x \pm r h_1 \phi(x)/2, y \pm r h_2 \phi(y)/2 \big) \in [-1,1]^2, $ $\overline{\Delta}_{h_1 \phi(x),h_2 \phi(y)}^r f(x,y)=0$ elsewhere.
\end{definition}
\begin{theorem} \cite{Gal}\label{BFOTHM7}
If $f$ is real-valued continuous on $[-1, 1]^2,$ a sequence of bivariate polynomials $(P_{m,n}(f))_{m,n \in \mathbb{N}}$ exists, with degree $\le m+n, $ such that  $$ \| f - P_{m,n}(f)\| \le C~~ \omega_2^{\phi}\Big(f;\frac{1}{ m},\frac{1}{ n}\Big),$$ where $C >0$ is independent of  $f, m$ and $n,$ and $\omega_2^{\phi}\Big(f;\frac{1}{m},\frac{1}{n}\Big)$ is the Ditzian-Totik modulus of smoothness with $ \phi(x)= \sqrt{1-x^2}.$

  \end{theorem}
  \begin{notation}
   We define
  $E_{m,n}(f):= \inf\{\|f- p\|_{\infty}: p \in \mathcal{P}_{m,n}(I \times J)\} $ and
  $E^{\alpha}_{m,n}(f):= \inf\{\|f- p^{\alpha}\|_{\infty}: p^{\alpha} \in \mathcal{P}_{m,n}^{\alpha} (I \times J)\}.$
\end{notation}
  \begin{theorem}
    If $f$ is real-valued continuous on $[-1,1]^2$, then the following estimate holds:

    $$ E^{\alpha}_{m,n}(f) \le C ~\frac{1+\|\alpha\|_\infty \big(\|Id-L\|-1 \big)}{1-\|\alpha\|_\infty} \omega_2^{\phi} \Big(f;\frac{1}{ m},\frac{1}{ n}\Big)+ \frac{\|\alpha\|_{\infty} \| Id - L\|}{1-\|\alpha\|_{\infty}}\|f\|_{\infty},$$ where $C >0$ is an absolute constant.

    \end{theorem}
    \begin{proof}
   Let $f \in \mathcal{C}([-1,1]^2,\mathbb{R}).$ Let $p_f \in \mathcal{P}_{m,n}([-1,1]^2)$ be a best approximant to $f.$ That is, $E_{m,n}(f)=\| f - p_f \|_{\infty} .$ Using the previous theorem, we  estimate a bound for $E_{m,n}^{\alpha}(f)$ in the following manner:

    \begin{equation*}
           \begin{aligned}
              E_{m,n}^{\alpha}(f) \le &~ \| f - (p_f) ^{\alpha}\|_{\infty}\\
               \le &~ \| f - p_f \|_{\infty} + \| p_f - (p_f)^{\alpha}\|_{\infty}\\
               \le &~ E_{m,n}(f) + \frac{\|\alpha\|_{\infty} \| Id - L\|}{1-\|\alpha\|_{\infty}}\|p_f\|_{\infty}\\
               \le &~ E_{m,n}(f) + \frac{\|\alpha\|_{\infty} \| Id - L\|}{1-\|\alpha\|_{\infty}} \|p_f-f+f\|_\infty\\
               \le &~ E_{m,n}(f)\Big[ 1+ \frac{\|\alpha\|_{\infty} \| Id - L\|}{1-\|\alpha\|_{\infty}}\Big]+ \frac{\|\alpha\|_{\infty} \| Id - L\|}{1-\|\alpha\|_{\infty}} \|f\|_\infty\\
               \le &~ \frac{1+\|\alpha\|_\infty \big(\|Id-L\|-1 \big)}{1-\|\alpha\|_\infty} E_{m,n}(f)+  \frac{\|\alpha\|_{\infty} \| Id - L\|}{1-\|\alpha\|_{\infty}} \|f\|_\infty
\end{aligned}
           \end{equation*}
    hence by the previous theorem, we obtain the result.
\end{proof}

\bibliographystyle{amsplain}

\end{document}